\documentclass[12pt]{amsart}

\usepackage{amscd}
\usepackage{srcltx}
\usepackage{amsmath}
\usepackage{amsfonts}
\usepackage{verbatim}
\usepackage{amsthm}
\usepackage{amssymb, mathrsfs}
\usepackage{amscd}
\usepackage{latexsym}
\usepackage{array}
\usepackage{enumerate}
\usepackage{longtable}
\usepackage[all]{xypic}
\usepackage[latin1]{inputenc}

\ifx\pdfpagewidth\undefined 
\usepackage[T1]{fontenc}
\else 
\usepackage[OT1]{fontenc} 
\fi 
\usepackage{amsfonts,amssymb,latexsym,longtable,amsmath}

\newtheorem{Theorem}{Theorem}[section]
\newtheorem{cor}[Theorem]{Corollary}

\newtheorem{Definition}[Theorem]{Definition}

\newtheorem{Example}[Theorem]{Example}

\newtheorem{Lemma}[Theorem]{Lemma}

\newtheorem{Proposition}[Theorem]{Proposition}
\newtheorem{theor}[Theorem]{Theorem}

\newcommand{\R}{\mathbb R}
\newcommand{\C}{\mathbb C}

\newcommand{\T}{\mathbb T}

\newcommand{\Q}{\mathbb Q}

\newcommand{\sB} {{\mathcal B}}

\newcommand{\Ga} {{\Gamma}}

\newcommand{\lra}{\longrightarrow}

\renewcommand{\phi}{\varphi}
\renewcommand{\theta}{\vartheta}

\newcommand{\irn}{\textrm{irrep}_n(G)}
\newcommand{\ir}{\textrm{irrep}(G)}
\newcommand{\irns}{\textrm{irrep}_n}

\newcommand{\rep}{\textrm{rep}}
\newcommand{\repg}{\textrm{rep}(G)}

\newcommand{\repng}{\textrm{rep}_{n}(G)}

\newcommand{\U}{\mathbb{U}}

\newcommand{\W}{\operatorname{w}}
\newcommand{\D}{\operatorname{d}}

\begin{document}
\title{Dual topologies on non-abelian groups}

\author[M. Ferrer]{M. Ferrer}
\address{Universitat Jaume I, Instituto de Matem\'aticas de Castell\'on,
Campus de Riu Sec, 12071 Castell\'{o}n, Spain.}
\email{mferrer@mat.uji.es}
\author[S. Hern\'andez]{S. Hern\'andez}
\address{Universitat Jaume I, Departamento de Matem\'{a}ticas,
Campus de Riu Sec, 12071 Castell\'{o}n, Spain.}
\email{hernande@mat.uji.es}

\thanks{ The first-listed
author acknowledges partial support by Generalitat Valenciana,
grant number GV-2009-021 and Fundaci\'o Caixa Castell\'o
(Bancaixa), grant P1.1B2008-09, the second-listed author
acknowledges partial financial support by the Spanish Ministry of
Science, grant MTM2008-04599/MTM; and by Fundaci\'o Caixa
Castell\'o (Bancaixa), grant P1.1B2008-26}
\subjclass{Primary 22C05, 22D35, 22A05, 43A40; Secondary 43A65, 54H11}
\keywords{compact group, maximally almost periodic group, totally bounded group,
$\kappa$-narrow group, $\kappa$-narrow uniform space, Lindel\"of number,
determined group, Tannaka-Kre\v{\i}n duality, dual topological group}
\dedicatory{Dedicated to Dikran Dikranjan}

\date{December/16/10}

\maketitle \setlength{\baselineskip}{24pt}
\setlength{\parindent}{1cm}
\begin{abstract}
The notion of locally quasi-convex abelian group, introduce by Vilenkin, is extended to
maximally almost-periodic non-necessarily abelian groups. For that purpose, we look at
certain bornologies that can be defined on the set $\hbox{rep}(G)$ of all finite dimensional continuous
representations on a topological group $G$ in order to associate well behaved group topologies
(dual topologies) to them. As a consequence, the lattice of all Hausdorff totally bounded group topologies
on a group $G$ is shown to be isomorphic to the lattice of certain special subsets of $\hbox{rep}(G_d)$.
Moreover, ge\-ne\-ra\-li\-zing some ideas of Namioka, we relate the structural properties of the dual topological
groups to topological properties of the bounded subsets belonging to the associate bornology.
In like manner, certain type of bornologies that can be defined on a group $G$ allow one to
define canonically associate uniformities on the dual object $\widehat G$. As an application,
we prove that if for every dense subgroup $H$ of a compact group $G$ we have that $\widehat H$
and $\widehat G$ are uniformly isomorphic, then $G$ is metrizable. Thereby, we extend to
non-abelian groups some results previously considered for abelian topological groups.
\end{abstract}
\vspace{1cm}
\section{Introduction and motivation}

In this paper we investigate the application of Pontryagin duality methods
to non-abelian groups and our mail goal is to study general groups using
these techniques. Although there are many ways of approaching non-abelian
duality, here, we shall explore a method that is based on the duality
theory of compact groups established by Tannaka and Kre\v{\i}n (see \cite{hr:ii,enoschw}
and the references therein).

Duality methods are well known and widely used in the study of topological
abelian groups and the literature on this subject is vast. This
program has not been accomplished for non-abelian groups and there
are several reasons that explain this asymmetric development
between both theories. No doubt, the non-abelian
context is more involved since the \emph{dual object} is not longer
a topological group and it seems difficult to equip it with a
structure which be sufficiently general and tractable.
Notwithstanding this, it is known, as a consequence of the cellebrated Gel'fand and
Ra\v{i}kov Theorem, that the set of all unitary representations of
a locally compact group $G$ contains the information necessary to
recover the topological and algebraic structure of the group (see
\cite{enoschw}). In this paper we look at the duality theory of
Maximally Almost Periodic groups (MAP groups for short) and their finite
dimensional representations. Our main goal is to develop this
theory in order to apply duality methods for non-abelian groups.
A main tool here is the unitary duality introduced by Tannaka, Kre\v{\i}n,
Hochschild and Mostow for compact groups,  and extended subsequently by
different authors \cite{chu,heyer70,pogu72,pogu76a,pogu76b,roed,gal_her:advances,
hart_kunen:BohrNonabelianGroups,hernwu06,rigg}.

In this line, the notion of locally quasi-convex abelian group, introduce by Vilenkin, is extended to
maximally almost-periodic non-necessarily abelian groups. For that purpose, we look at
certain bornologies that can be defined on the set $\hbox{rep}(G)$ of all finite dimensional continuous
representations on a topological group $G$ in order to associate well behaved group topologies
(dual topologies) to them. As a consequence, the lattice of all Hausdorff totally bounded group topologies
on a group $G$ is shown to be isomorphic to the lattice of certain special subsets of $\hbox{rep}(G_d)$.
Moreover, genera\-li\-zing some ideas of Namioka, we relate the structural properties of the dual topological
groups to topological properties of the bounded subsets belonging to the associate bornology.
In like manner, certain type of bornologies that can be defined on a group $G$ allow one to
define canonically associate uniformities on the dual object $\widehat G$. As an application,
we prove that if for every dense subgroup $H$ of a compact group $G$ we have that $\widehat H$
and $\widehat G$ are uniformly isomorphic, then $G$ is metrizable. Thereby, we extend to
non-abelian groups some results previously considered for abelian topological groups.

\section{Basic definitions and terminology}
All (topological) groups are assumed to be (Hausdorff
and) maxi\-mally almost periodic (MAP, for short); that is, groups that can be
continuously injected into compact groups. For any group $G$, let $e_G$ denote the neutral element
of the group.

Let $\mathbb{U}(n)$ denote the
{\em unitary group of order $n$}, namely, the group of all
complex-valued $n\times n$
matrices $A$ for which $A^{-1}=A^{*}$.
Then $\mathbb{U}(n)$ is a compact Lie group and can be
realized as the group of isometries of $\C^n$. We endow
$\mathbb{U}(n)$ with the metric coming from the operator
norm. More specifically
$$d(\varphi,\psi)=\|\varphi-\psi\|=\|\varphi \psi^{-1} -I_n\|$$
Set $\mathbb{U}=\sqcup _{n<\omega }\mathbb{U}(n)$
(topological/uniform sum).

A {\em unitary representation} $\varphi$ of the (to\-po\-lo\-gi\-cal) group $G$
is a (continuous) homomorphism into the group of all linear
isometries of a complex Hilbert
space $\mathcal{H}$.
When dim~$\mathcal{H}<\infty$, we say that $\varphi$ is a {\it finite
dimensional representation}; in this case, $\varphi$ is a homomorphism
into one of the groups $\mathbb{U}(n)$.

A linear subspace $E\subseteq \C^n$ is an \emph{invariant} subspace
for $\mathcal S\subseteq \mathbb{U}(n)$ if $ME\subseteq E$ for all
$M\in \mathcal S$.
If there is a subspace $E$ with $\{0\}\subsetneq E\subsetneq \mathbb C^n$
which is invariant for $\mathcal S$, then $\mathcal S$ is called
\emph{reducible}; otherwise $\mathcal S$ is \emph{irreducible}.
An \emph{irreducible representation} of $G$ is a 
unitary representation $\varphi$ such that $\hbox{ran}(\varphi)$ is irreducible.

If $G$ is a compact group, the Peter-Weyl Theorem (see \cite{hof_mor:compact_groups}) implies that the finite
dimensional irreducible representations of $G$ yield a complete analysis of $G$.
The set of all $\mbox{$n$-dimensional}$
continuous unitary representations on $G$  will be denoted by
$\repng$ and the symbol $\repg=\bigcup\limits_{n<\omega}\ \repng$ denotes the set of all finite dimensional
continuous unitary representations.
In like manner, the set of all irreducible $n$-di\-men\-sional
continuous unitary representations on $G$  will be denoted by
$\irn$ and $\ir =\bigcup\limits_{n<\omega}\ \irn$.

Two $n$-dimensional representations $\varphi$ and $\psi$
are {\it equivalent} ($\varphi \sim \psi$)  when
there is $M$ in $\mathbb U(n)$ such that $\varphi(x)=M^{-1}\psi(x)M$ for all
$x\in G$. This defines an equivalence relation on the set of all $n$-dimensional
representations of the group. The
quotient set $\widehat{G}_n=\frac{\irn}{\sim}$  consists
of the equivalence classes of irreducible representations of
dimension $n$.
Denote by $\widehat G=\bigcup\limits_{n<\omega} \widehat{G}_n$ the collection of all equivalence
classes $[\varphi]$ of irreducible representations $\varphi$ of $G$.
One calls $\widehat G$ the \emph{dual object} of $G$.
Following von Neumannn \cite{vneumann}, in order to deal with $\widehat G$,
one must make an arbitrary
choice of an irreducible representation in each equivalence class.

When be convenient, the symbol $G_\tau$ will
denote the group $G$ equipped with a group topology $\tau$ and,
more generally, the notions related to $G_\tau$ must be understood
in this context.
For example, the set of all
\emph{continuous} $n$-di\-men\-sio\-nal unitary representations
on $G_\tau$ will be denoted by $\hbox{rep}_n(G_\tau)$ and so on. The symbol
$G_d$ will denote the algebraic group $G$ equipped with the discrete topology.

We now collect some well-known facts about uniform spaces. The basic definitions and terminology are
taken from \cite{roelcke_dierolf}.

Let $X$ be a set and $\Delta= \Delta(X)\subseteq X\times X$ the
diagonal on $X$. For $B,C\subseteq X\times X$, $S\subseteq X$ we
define $B\circ C= \{(x,z)\in X\times X : (x,y)\in B \ \hbox{and} \
(y,z)\in C \ \hbox{for some} \ y\in X \}$, $B^{-1}= \{(x,y) :
(y,x)\in B \}$, $B[S]=\{y\in X : (x,y)\in B \ \hbox{for some} \
x\in S \}$. A {\it uniformity} on $X$ is a set $\mu$ of subsets of
$X\times X$ satisfying the following conditions: (i) $B\supseteq
\Delta$ for all $B\in \mu$; (ii) if $B\in \mu$, then $B^{-1}\in
\mu$; (iii) if $B\in \mu$, there exists $C\in \mu$ such that
$C\circ C\subseteq B$; (iv) the intersection of two members of
$\mu$ also belongs to $\mu$; (v) any subset of $X\times X$, which
contains a member of $\mu$, itself belongs to $\mu$; (vi) $\Delta=\cap\{B: B\in\mu\}$.
If the latter condition is not satisfied we say that $\mu$ is a \emph{pseudo-uniformity} on $X$.

The members of $\mu$ are called {\it bands} (of the uniformity).
By a {\it base} for a uniformity $\mu$ is meant a subset $\mathfrak
B$ of $\mu$ such that a subset of $X\times X$ belongs to $\mu$ if
and only if it contains a set belonging to $\mathfrak B$. A {\it
uniform space} $\mu X$ is a pair comprising a set $X$ and a
uniformity $\mu$ on $X$. If $\mu X$ is a uniform space, one may
define a topology $\tau$ on $X$ by assigning to each point $x$ of
$X$ the neighborhood base comprised of the sets $B[x]$, $B$
ranging over the uniformity. The
topological space associated to $\mu X$ will be denoted by $X_{\tau\mu}$.
Given two covers $\mathcal U$ and $\mathcal V$ of $X$ we say that
$\mathcal U$ \emph{refines} $\mathcal V$ if for all $V\in \mathcal
V$ there is $U\in \mathcal U$ such that $U\subseteq V$. A cover
$\mathcal{U}$ of $X$ is named \emph{$\mu$-uniform} if
there exists $B\in\mu$ such that the cover $\{B[x]\,:\,x\in X\}$
refines $\mathcal{U}$.

A topological group $G$ is said to be \emph{totally bounded} when
for every neighborhood $U$ of the neutral element in $G$, there is
$F\subseteq G$ with $|F|< \aleph_0$ such that
$FU=G$. This notion has also meaning for arbitrary cardinal numbers in
the following manner:
let $\kappa$ be a cardinal
number, following the terminology of Arhangel'skii and Tkachenko (see \cite{arh_tka:book}),
we say that $G$ is \emph{$\kappa$-narrow} when for every neighborhood $U$
of the neutral element in $G$, there is $F\subseteq G$ with $|F|\leq \kappa$ such that
$FU=G$. These definitions extend to uniform spaces in a natural way.
\begin{Definition}
\rm{ Let $\mu X$ be a uniform space,
we say that $\mu X$ is \emph{totally bounded} (resp. \emph{$\kappa$-narrow})
when every $\mu$-uniform cover of $X$ has a subcover of finite cardinality (resp.
of cardinality less or equal than $\kappa$). The \emph{Lindel\"of number} of a
uniform space $\mu X$ is defined
as the less cardinal $\kappa$ such that $\mu X$ is $\kappa$-narrow.}
\end{Definition}

It is known that a uniform space $\mu X$ is $\kappa$-narrow if and only
if it contains no uniformly discrete subspace of power $\kappa$
(see \cite[p. 24]{isbell:book}). From this fact, it follows:
\begin{Proposition}
Every (uniform) subspace of a $\kappa$-narrow uniform space
is $\kappa$-narrow.
\end{Proposition}

In this paper, a \emph{bornology} $\mathcal B$ on a space $X$
designates an ideal of subsets of $X$ that contains all finite subsets
of $X$ (denoted $Fin(X)$). That
is to say, $\mathcal B$ satisfies the following properties:

\begin{itemize}
\item[(i)] $A, B\in \mathcal B$ implies $A\cup B\in \mathcal B$;

\item[(ii)] $A\in \mathcal B$ and $B\subseteq A$ implies $B\in
\mathcal B$;

\item[(iii)] $Fin(X)\subseteq\mathcal B$.
\end{itemize}

Let $X$ be a set and let $M$ be a metrizable space. We
denote by $M^X$ (resp. $l^\infty(X,M)$) the set of all (resp. \emph{bounded})
functions from $X$ to $M$. Suppose that $\mathcal B$ is a bornology
on $X$ and $F\subseteq M^X$ is pointwise bounded. Then we may equip
$F$  with the
uniformity $\mu_{\mathcal B}$, which has as a uniformity base  the bands
$$B(K,\epsilon)=\{(f,g)\in F^2\,:\,d(f(x),g(x))<\epsilon,
\,\forall x\in K\}$$ where $\epsilon>0$, $K\in \mathcal B$
and $d$ is the metric defined on $M$.

If $F\subseteq l^\infty(X,M)$, there also exists a metric uniformity
$\mu_{\infty}$ that can be
defined canonically on $F$, which is generated by the bands
$\{ B(X,\epsilon): \epsilon > 0 \}$. The metric associated to
$\mu_{\infty}$ is defined for all $(f,g)\in F^2$ by
$$d_{\infty}(f,g):=\sup\limits_{x\in X}d(f(x),g(x)).$$

\noindent We assume that $l^\infty(X,M)$ is equipped with this metric in principle.

From here on, we shall make use of the \emph{evaluation mapping} $\mathfrak e_F$
that is associated to every $F\subseteq M^X$ as
follows:
$$\begin{array}{rcl}\mathfrak e_F:X&\hookrightarrow & M^F\\
x & \mapsto & \mathfrak e_F(x)(f)=f(x)\end{array}$$

Furthermore, if $F$ is pointwise bounded, it follows that
$\mathfrak e_F(X)\subseteq l^\infty(F,M)$. Therefore,
we may equip $X$ with the pseudo-metric
$d_F$ inherited from $l^\infty(F,M)$, which is defined by
$$d_F(x,y)=\sup\limits_{f\in
F}d(f(x),f(y))\ \forall (x,y)\in X^2.$$

The pseudo-uniformity associated to $d_F$ is denoted by $\mu_F$.
The symbol $X_F$ denotes the space $X$ equipped with the pseudo-metric
$d_F$. We also identify $X_F$ to $X_{\tau\mu_F}$ for the sake of simplicity.

Given a topological space $X$ and a metrizable space $M$, let $C(X,M)$ be the set
of all  continuous functions from $X$ to $M$. We denote
by $\mu_\mathcal K$ the uniformity on
$C(X,M)$ generated by the bornology $\mathcal B=\mathcal K(X)$
generated by all compact subsets of $X$.

\section{Variations on a theorem by Namioka}

The starting point of this section is the following generalization of a theorem by Namioka in \cite{namioka:mathematika}
about the relationship between the metrizability and separability of certain spaces. In order to
do this, we need some preliminary notational observations.

Let $X$, $F$ and $M$ be sets and let $\Phi : X\times~F \longrightarrow M$ be an arbitrary map.
Identifying the elements $f,g$ in $F$ such that $\Phi(x,f)=\Phi(x,g)$ for all $x\in X$, we assume
without loss of generality that every element $f\in F$ is univocally determined by the map
$\Phi_f\in M^X$, defined as $\Phi_f(x)=\Phi(x,f)$. Therefore, with some notational abuse, the
set $F$ may be looked at as a subset of $M^X$. With this assumption the symbol $cl_{M^{X}} F$
denotes the set $cl_{M^X} \{ \Phi_f : f\in F \}$.

The \emph{density} $\D(X)$ of a topological space $X$ is the minimal cardinality of a dense subset of $X$.
The \emph{weight} $\W(X)$ is the minimal cardinality of an open basis for the topology of $X$.
We have $\W(X)=\D(X)$ for infinite metrizable spaces $X$.
\begin{theor}\label{th_npeso}
Let $X$, $F$ be sets and let $M$ be a metrizable space.  If $\Phi : X\times~F \longrightarrow M$ is a map such that $K=cl_{M^{X}} F$ is
compact. Then the following assertions are true:
\begin{itemize}
\item[(a)] If $\D(X_F)\leq \kappa$, then $\W(K)\leq \kappa$.

\item[(b)] If $\W(M)\cdot \W(K)\leq  \kappa$, then
$\D(X_F)\leq \kappa$.
\end{itemize}
\end{theor}
\begin{proof}[(Sketch)] (a) Because $X_F$ is equipped with the topology of uniform convergence,
it follows that $F$ is equicontinuous. Therefore also $K$ is equicontinuous on $X$. Thus any dense
subset $A$ of $X_F$ with $|A|\leq \kappa$ separates the points of $K$. The proof of (b) is clear.
\end{proof}

Next we present several consequences that follow from the result above. In the rest of this section,
$F$ denotes a pointwise bounded subset of $M^X$.

\begin{cor}
Let $\mu X$ be a $\kappa$-narrow uniform space and let $M$ be a metrizable space.
If  $F$ is a uniformly equicontinuous and relatively
pointwise compact subset of\ $M^X$,  then
$\W (F)\leq \kappa$.
\end{cor}
\begin{proof}
Observe that, if $F$ is uniformly equicontinuous, then
$\mathfrak e_F : \mu X\longrightarrow l^\infty(F,M)$ is uniformly
continuous (indeed, for every $\epsilon >0$ there exists
$B_{\epsilon}\in\mu$ such that if $(x,y)\in B_{\epsilon}$ then
$d(f(x),f(y))<\epsilon$ $\forall f\in F$, that is,
$(\mathfrak e_F(x),\mathfrak e_F(y))\in B(X,\epsilon)$). Therefore the
uniformity $\mu$ is finer than $\mu_F$, which implies that $\mu_F
X $ is $\kappa$-narrow. Since $\mu_F X$ is
pseudo-metrizable, it follows that $\D(X_F)\leq \kappa$.
It will suffice now to apply Theorem \ref{th_npeso} to the map
$\Phi(x,f)=f(x)$ in order to finish the proof.
\end{proof}


Assume further that $F$ is equipped with some bornology $\mathcal B$ and
$X$ is equipped with the uniformity of uniform convergence on the members of $\mathcal B$.
Firstly, we need a technical lemma that we think it is probably known. Nevertheless,
we include a proof of it here because it is essential in the rest of the paper.
\begin{Proposition}\label{pro_pro}
The product of an arbitrary family of $\kappa$-narrow
uniform spaces is a $\kappa$-narrow
uniform space.
\end{Proposition}
\begin{proof}
Since a product has the weak uniformity generated by projections,
it will suffice to prove that the product of two $\kappa$-narrow uniform spaces is $\kappa$-narrow.
Now, let $\mu X$ and $\nu Y$ be $\kappa$-narrow. If $\mathcal U$ is a $(\mu\times \nu)$-uniform
cover of $X\times Y$, then $\mathcal U$ must have a refinement of the form
$\mathcal U_1\times \mathcal U_2= \{U_1\times U_2 : U_i\in \mathcal U_i,
1\leq i\leq 2 \}$, where $\mathcal U_1$ is $\mu$-uniform and $\mathcal U_2$ is
$\nu$-uniform. By hypothesis there is a subcover $\mathcal V_i$
of $\mathcal U_i$ such that $|\mathcal V_i|\leq\kappa$, $1\leq i\leq 2$.
Then $\mathcal V_1\times \mathcal V_2$ is a $(\mu\times \nu)$-uniform refinement
of $\mathcal U_1\times \mathcal U_2$  and
has cardinality less or equal than $\kappa$. This completes the proof.
\end{proof}
\begin{Theorem}\label{th_alfa_peso}
Let $X$ be a set, let $M$ be a metrizable space, and let $Y$ be a subset of $M^X$
that is equipped with some bornology $\mathcal B$ consisting of pointwise relatively compact sets.
If $\mu_\mathcal B$ is the uniformity on $X$ defined by $\mu_\mathcal B=\sup\{\mu_F\,:\,F\in
\mathcal{B}\}$ and furthermore $\W(M)\leq \kappa$, then $\mu_\mathcal B X$ is $\kappa$-narrow
if and only if $\W(F)\leq \kappa$ $\forall F\in\mathcal{B}$.
\begin{proof}
\noindent ($\Rightarrow$) If $\mu_\mathcal B X$ is $\kappa$-narrow then every $\mu_F
X$ is also $\kappa$-narrow and, since $\mu_F X$ is pseudo-metrizable,
it follows that $\D(X_F)\leq \kappa$. By Theorem \ref{th_npeso}, we obtain $\W(F)\leq\kappa$
for all $F\in \mathcal B$.

\noindent ($\Leftarrow$) Since $\W(F)\leq \kappa$ for all $F\in \mathcal B$, we deduce by
    Theorem \ref{th_npeso} again that $\D(X_F)\leq \kappa$ for all $F\in \mathcal B$.
    Thus $\mu_F X$ is $\kappa$-narrow for all $F\in \mathcal B$. Since
$\mu_\mathcal B X$ is the diagonal (uniform) subspace of $\prod\limits_{F\in \mathcal B} \mu_FX$,
it follows by the preceding proposition  that $\mu_\mathcal B X$ is $\kappa$-narrow too.
\end{proof}
\end{Theorem}

As a consequence, we obtain the following corollary for general topological spaces.

\begin{cor}
If $X$ is a $\kappa$-Lindelöf space and $F$ is an equicontinuous subset
of  $C(X,M)$ then $\W(F)\leq \kappa$.
\end{cor}

\section{Dually Representable MAP Groups}

Assume that our overall goal now is to
study the group topologies that can be defined using
finite dimensional representations. Because we are interested in Hausdorff
topologies, it will suffice to deal only
with groups whose finite dimensional representations separate points (MAP groups).
Firstly, it is necessary to remember
which collections of representations
identify a set of \emph{continuous} representations for
some group topology. Our approach here is based on the celebrated Tannaka-Kre\v{\i}n
duality, where we have followed the well know text by Hewitt and Ross
\cite{hr:ii} although this theory is currently treated using categorical
methods. We recommend the article by Joyal and Street \cite{joy_str:tannaka}
that is illuminating in this regard.


Let $G$ be a topological group and $\Gamma\subseteq \repg$ be a
separating subset, that is, for all $e_G\not= x\in G$ there
exists $\varphi\in \Gamma$ such that $\varphi(x)$ is not the identity matrix.
If $\Gamma_n$ denotes the set of all $n$-dimensional elements in
$\Gamma$, we have $\Gamma=\bigcup\limits
_{n < \omega}\Gamma_n$ and we say that $\Gamma$ is a
\emph{representation space} for $G$ if  it satisfies:
\begin{enumerate}
\item $\varphi_1\in \Gamma_n,\varphi_2\in \Gamma_m\Rightarrow \varphi_1\oplus \varphi_2\in
\Gamma_{n+m},\, \varphi_1\otimes \varphi_2\in \Gamma_{nm}$; \item if $\varphi\in \Gamma_n$ and $M\in
\mathbb U(n)$, then $\bar \varphi$ and $M^{-1}\varphi M$ belong to $\Gamma_n$;
\item if $\Gamma\ni \varphi=\oplus_{j=1}^m \varphi_j$ then $\varphi_j\in
\Gamma$ for all $1\leq j\leq m$; \item $1\in \Gamma_1$.
\end{enumerate}

We notice that (4) follows from (1)-(3). It is also possible to define representation spaces using
only equi\-va\-lence classes of irreducible representations.
Moreover, the notion of \emph{Kre\v{\i}n} algebra makes possible to define
the notion of representation space without referring to any specific group
$G$.

A bornology $\mathcal B$ of subsets in $\Gamma$ (and consequently
$\mathcal B_n=\mathcal B \cap \mathcal P(\Gamma_n)$ is a bornology in $\Gamma_n$)
is named  \emph{(representation) bornology} when satisfies the following
properties:
\begin{enumerate} \item for every $P\in \mathcal B$, there is a
natural number $N_P$ such that $P=\bigcup\limits_{n<N_P}P_n$, with $P_n\in \mathcal B_n$
and some $P_n$ may be the empty set;
\item if $P\in \mathcal{B}_n,Q\in
\mathcal{B}_m\Rightarrow P\oplus Q\in \mathcal{B}_{n+m},\,
P\otimes Q\in \mathcal{B}_{nm}$.
\end{enumerate}
The pair $(\Gamma,\mathcal B)$ allows us to equip
$G$ with a topology $t(\mathcal B)$ that has as a neighborhood base
at the identity the family of sets $[P,\epsilon]$ where
\begin{enumerate}
\item $[P,\epsilon]=G$ if $P=\emptyset$;
\item $[P,\epsilon]=\{g\in G:  \|\varphi(g)-I_n\|\leq\epsilon\ \mbox{for all}\ \varphi\in
P\}$ if $\emptyset\not= P$ belongs to $\mathcal{B}_n$; and
\item $[P,\epsilon]=\bigcap\limits_{n<N_P} [P_n,\epsilon]$ if $P=\bigcup\limits_{n<N_P}P_n\in \mathcal B$.
\end{enumerate}

\noindent Here $I_n$ is the identity matrix in $\mathbb{U}(n)$ and $\epsilon>0$ for every $n\in \mathbb{N}$.

For instance, we
shall make us of the topology $t(Fin(\repg))$ associated to
the bornology $Fin(\repg)$ consisting of all finite subsets of $\repg$. In case
that $G$ is a compact group, the celebrated Peter-Weyl Theorem establishes that
the groups $G$ and $G_{t(Fin(\repg))}$ are topologically isomorphic. Observe that
the  bornology $Fin(\repg)$ defines the initial, or weak, (group) topology
generated by $\repg$. In this case, we may say that $Fin(\repg)$ is the smallest bornology on $\repg$.
Moreover, for any algebraic group $G$, there  exists a biggest
bornology $\mathcal B^+$ that is generated
by the collection $\{ \rep_n(G_d) : n<\omega \}$.
The group $G_{t(\mathcal B^+)}$ is not necessarily discrete as it will be seen below
 (cf. \cite{hernwu06}). If $G_\tau$ is a topological group and $\mathcal E$
 denotes the bornology of equicontinuous subsets on $\repg$, then
 $t(Fin(\repg))\subseteq t(\mathcal E)\subseteq \tau$.

The proof of the following result is clear.
\begin{Proposition}\label{pr_top}
$(G,t(\mathcal{B)})$ is a Hausdorff topological group.
\end{Proposition}

A topology $t(\mathcal{B})$ on a group $G$ which is associated
to a bornology $\mathcal B$ on a representation space $\Gamma$ is
named \emph{dually induced topology} and we say that
$G_{t(\mathcal{B})}$ is a \emph{dually induced topological
group}. This notion has been extensively studied for abelian
groups and it turns out that most important classes of topological abelian groups
possesses dually induced topologies. The main goal of this paper is the
understanding of the dually induced group topologies for
non-abelian MAP groups.

It is not difficult to verify that for every group $G$ and
representation space $\Gamma$, the bornology $Fin(\Gamma)$ defines
a totally bounded topology on $G$. Conversely, we show next that every
totally bounded group topology is a dually induced topology. This our
version of a a similar result given by Comfort and Ross for
totally bounded abelian groups (cf. \cite{com_ros:BoundedTopologies}).

\begin{Theorem}
For any group $G$, the lattice of (Hausdorff) totally bounded group topologies
on $G$ is bijectively isomorphic to the lattice of representation spaces of $G$.
\end{Theorem}
\begin{proof}
Let $\mathcal{HTB}(G)$ denote the lattice of Hausdorff totally bounded group topologies
on $G$ and let $\mathcal{RS}(G)$ be the lattice of representation spaces on $G$.
We define $$\Phi : \mathcal{HTB}(G)\longrightarrow \mathcal{RS}(G)\ \hbox{by}\ \Phi(\tau)=\hbox{rep}(G_\tau)$$
and $$\Psi : \mathcal{RS}(G)\longrightarrow \mathcal{HTB}(G)\ \hbox{by}\ \Psi(\Gamma)=t(Fin(\Gamma)).$$
Clearly, these maps are well defined because the groups $\U(n)$ are compact.
It will suffice to verify that $\Phi\circ \Psi$ and $\Psi\circ \Phi$ coincide with the identity mapping in order
to finish the proof. Suppose that $\tau\in \mathcal{HTB}(G)$ and let $K$ be the
Weil completion of $G_\tau$, that is a compact group. The groups $K$
and $G_\tau$ possess the same set of continuous representations $\rep(G_\tau)$
and, by Peter-Weyl Theorem, it follows that
$K$ is equipped with the topology $t(Fin(\rep(G_\tau)))$. Therefore,
since $G_\tau$ is dense in $K$, we have $\tau=t(Fin(\rep(G_\tau)))$, which shows
that $\Psi\circ \Phi$ is the identity map. Now, let $\Gamma\in \mathcal{RS}(G)$ and
set $\tau=t(Fin(\Gamma))=\Psi(\Gamma)$. Again, let $K$ be the compact group obtained
as the Weil completion of $G_\tau$. By continuity, it is clear that every element $\varphi\in \Gamma$
can be extended to a continuous representation $\varphi^K$ on $K$. Thus, we have that
the representation space $\Gamma^K=\{\varphi^K : \varphi\in \Gamma \}$ is canonically included in
$\hbox{rep}(K)$. Moreover, if $e\not=x\in K$, there is a neighbourhood of the neutral element $U$
such that $x\notin U^2$. Take $F\in Fin(\Gamma)$ and $\epsilon >0$ such that
$[F,\epsilon]\subseteq U\cap G$. Using a standard continuity argument, it is readily seen
that $\|\varphi^K(x)-I\|\geq \epsilon/2$. Thus the set $\Gamma^K$ satifies (1)-(4) and
separates the points in $K$. By Tannaka duality this means that $\Gamma^K=\hbox{rep}(K)$
(see \cite{joy_str:tannaka}). Since $K$ and $G_\tau$ posess the same continuous representations,
it follows that $\Phi\circ \Psi$ is the identity and this completes the proof.
\end{proof}

\begin{cor}
The bornology $Fin(\rep(G_d))$ gives the largest totally bounded
group topology on $G$.
\end{cor}

Unfortunately, the situation is not so nice for more general
classes of groups. In the following example (due to Moran) it
is shown that there
are discrete MAP groups that are not dually induced topological
groups.

Let $(p_n)$ be an infinite
sequence of distinct primes numbers ($p_n>2$) and let $G_n$ be
the projective special linear group of dimension two over the
Galois field $GF(p_n)$ of order $p_n$.
If $G=\prod\limits_{n=1}^{\infty}G_n$ equipped with discrete topology, then
$G_{t(\mathcal B^+)}=\prod\limits_{n=1}^{\infty}G_n$ (with product topo\-lo\-gy).

In general, if $G$ is a compact group such that $bG_d=G$, then $G_d$ is not a dually induced topological group.
Indeed, it is known (see
\cite{hart_kunen:BohrNonabelianGroups,gal_her:advances}) that, for
these groups, $\rep(G_d)=\rep(G)$ and moreover each $(\widehat{G_{d}})_n$ is finite
for all $n<\omega$.
As a consequence, the same is true for any maximal subset of pairwise non-equivalent collection
of $n$-dimensional irreducible representations, say $\Gamma(G_d)_n$. Since every finite-dimensional
representation is a direct sum of irreducible representations, this means that the group
$G_{t(\mathcal B^+)}$ associated to the biggest bornology
$\mathcal B^+$ on (the algebraic group) $G$ 
is topologically isomorphic to $G_{t(Fin(\repg))}$, which is topologically
isomorphic to the compact group $G$ by Peter-Weyl Theorem.
This implies that there is just one dually induced
topology on the (algebraic) group $G$; namely, its compact group topology.
Therefore, the group $G_d$ is not a dually induced topological group.

The situation for finitely generated discrete groups has been clarified completely
by A. Thom \cite{thom}.
\begin{Theorem}[Thom, 2010]
A finitely generated discrete group $G$ is a dual topological group if and only if
$G$ is abelian by finite.
\end{Theorem}

We do not know in general which discrete groups are dual topological groups.
But we notice that there are infinite discrete groups which are dual topological groups (see \cite{hernwu06}).
We denote by $exp(G)$ the exponent of a group $G$.
\begin{Theorem}\label{co39}
Let $G=\sum\limits_{i\in I}F_i$, where $F_i$ is a finite simple
non-abelian group for each $i\in I$. Then the discrete group $G$ is dual if
and only if the set $\{exp(F_i) : i\in I \}$ is bounded.
\end{Theorem}

The examples above show that it is necessary to consider infinite
dimensional representations if one wants to apply duality methods
for \emph{all} locally compact groups. Our methods here only deal with
groups whose topology may be understood using finite dimensional
representations.

We will now look at several methods of defining  dually induced
topologies. In the sequel, $G$ is a MAP group, $\Gamma$
a representation space for $G$ and $\mathcal B$ is a bornology
on $\Gamma$.
For any cardinal number $\kappa$, we define $\mathcal{B}_{\kappa}$
to be the ideal of subsets generated by the collection
$\{F\in \mathcal B : \W(F_{t(Fin(G))})\leq\kappa
\}.$
It is readily seen that $\mathcal{B}_{\kappa}$ is a bornology on
$\Gamma$. 
 As a
consequence of Theorem \ref{th_alfa_peso}, we obtain:

\begin{Proposition}\label{pr_lindelof}
$G_{t(\mathcal{B}_\kappa)}$ is $\kappa$-narrow.
\end{Proposition}

Now, we are going to extend this approach in order to establish the existence of some dually
induced topologies associated to every topological group.

\begin{Theorem}
Let $G$ be a MAP group. Then the following assertion are
satisfied.

\begin{enumerate}

\item[(i)] For every cardinal $\kappa$, there is a largest, dually
induced, $\kappa$-narrow topology $t(\kappa)$ on
$G$. In particular, there is a largest dually induced topology on
$G$.

\item[(ii)] If $\tau$ is a group topology on $G$, then there is a
largest, dually induced, $\kappa$-narrow  topology
$\tau_{\kappa}$ that is included in $\tau$. In particular, there
is a largest dually induced topology that is included in $\tau$.
\end{enumerate}
\begin{proof} (i) It suffices to apply Proposition \ref{pr_lindelof}
to the representation space $\Gamma=\rep(G_d)$ and the largest
bornology $\mathcal B^+$ generated by the
collection $\{ \rep_n(G_d): n< \omega \}$. By Theorem \ref{th_alfa_peso}, the
topology $t(\mathcal B^+_\kappa)$ is $\kappa$-narrow and
furthermore it is the finest $\kappa$-narrow, dually induced, topology
on $G$. On the other hand, the biggest bornology  $\mathcal B^+$,
induces the topology  $t^+$  that is the largest dually induced topology on $G$.

(ii) Let $\kappa$ be a cardinal number and let $\Gamma_\tau=\rep(G_\tau)$.
On $\Gamma_\tau$ consider the bornology $\mathcal E(\Gamma_\tau)_\kappa$ generated by the collection
of all equicontinuous subsets $A$ of $(\Gamma_\tau)_n$ for some $n<\omega$ and
$\W(A_{t(Fin(G))})\leq\kappa$. These defines a dually induced topology $t(\mathcal E(\Gamma_\tau)_\kappa)$
on $G$ that is weaker than $\tau$.

On the other hand, the collection of
all equicontinuous subsets $A$ of $\Gamma_\tau$ such that
$A\subseteq (\Gamma_{\tau})_n$ for some $n<\omega$ defines a
bornology $\mathcal E(\Gamma_\tau)$ on $\Gamma_\tau$ such that the topology
$\tau^+$, induced by $\mathcal E(\Gamma_\tau)$ on $G$, is the largest dually
induced topology on $G$ that is weaker than $\tau$.
\end{proof}
\end{Theorem}

We recall that if $G$ is an abelian topological group, the dual object $\widehat G$ is also a
topological group which coincides with $\hbox{irrep}(G)=\hbox{irrep}_1(G)=CHom(G,\T)$.

Let $G$ be an abelian group with no topology initially assumed on it.
Given a subgroup $\Gamma$ of $Hom(G,\T)$,
it is said that a topology $\tau$ on $G$ is \emph{compatible} with the pair $\langle G, \Gamma \rangle$
when $\widehat{G_\tau}=\Gamma$ algebraically.
In recent times, there is an interest in finding compatible topologies
for a given pair $\langle G, \Gamma \rangle$. We notice that if $\Gamma = \widehat{G_\tau}$ for some topology
$\tau$ on $G$, then $t(Fin(\Gamma))\subseteq t(\mathcal E(\Gamma)_\kappa)\subseteq \tau$
for each cardinal number $\kappa$.
Since the largest and the coarsest topologies in this chain have the same algebraic dual group,
it follows that the topologies $t(\mathcal E(\Gamma)_\kappa)$
are all compatible with $\langle G, \Gamma \rangle$.

Now, suppose that $G$ is a \v{C}ech-complete group (for instance, locally compact
or complete metrizable) and let $\Gamma=\widehat{G}$.
It is known that every pointwise compact subset of $\Gamma$ must be an equicontinuous on $G$
(see \cite{gal_her:fundamenta}).
Therefore, we obtain the following result.
\begin{Proposition}
Let $G$ be an abelian \v{C}ech-complete group and let $\Gamma=\widehat{G}$. If $\mathcal B$ is
the bornology of all pointwise compact subsets of $\Gamma$, then the dually induced
topologies $t(\mathcal B_\kappa)$ are compatible with $\langle G, \Gamma \rangle$ for all cardinal $\kappa$.
\end{Proposition}

The latter results extend in a straightforward manner to non-abelian groups. In this case, it is said
that a topology $\tau$ on $G$ is \emph{compatible} with the representation space $\Gamma$
when $\hbox{rep}(G_\tau)=\Gamma$.

\section{Uniformities on the representation spaces}

In this section, we look at the uniformities on the dual object of a topological group that are associated to certain
bornologies defined on it. Let $G$ be a (topological) group and let $\mathcal B$ be a bornology  on (the set) $G$.
We say that $\mathcal B$ is a \emph{group bornology} on $G$ when satisfies
the following properties:

\begin{enumerate}
\item [(1)]if $A\in \mathcal B$ then $A^{-1}\in \sB$;

\item [(2)] if $A,B$ belong to $\mathcal B$ then $AB$ belongs
to $\sB$.
\end{enumerate}

\noindent In most cases  $\mathcal B$ means the collection of all
relatively compact (resp. finite, precompact) subsets of $G$.

Assume now that $\Gamma$ is a representation space for $G$ and
$\mathcal B$ is a group bornology on the group. Since the
groups $\mathbb U(n)$ are compact, it follows that
$\Ga_n\subseteq C(G,\U(n))\cap l^\infty (G,\U(n))$. Therefore, we may equip $\Ga _n$ with
the uniformity $\mu_{\sB}$ associated to $\sB$ as described in Section~2.
The representation space $\Ga$ is now equipped with the free
uniformity sum of the family $\{\mu_\sB\Ga_n: n<\omega\}$.
We denote by $\mu_\sB\Ga$ the set $\Ga$ equipped with this
uniformity. 

The procedure introduced above allows us to define also a uniformity
on the dual space $\widehat G$ associated to a bornology $\mathcal B$
on $G$. Indeed, if $\mathcal B$ is a bornology on $G$, the symbol $\mu_\mathcal B \widehat G_n$
designates the set $\widehat G_n$ equipped with the final uniformity
induced by the canonical quotient  $Q_n:\mu_{\mathcal{B}}\irn \lra \widehat{G}_n$
and $\mu_\mathcal B\widehat G$ is the dual space equipped with the free
uniformity sum of the family $\{\mu_{\mathcal B}\widehat G_{n}: n<\omega\}$.
In particular, if $G$ is equipped with some group bornology $\sB$
and  $\kappa$ is a cardinal number, the family
$\mathcal{B}_{\kappa}:=\{K\subseteq \sB\,:\, \W(K_{t(Fin(\repg))})\leq\kappa\}$
defines a group bornology on $G$ and
$\mu_{\mathcal B_\kappa}\widehat G$ denotes the dual space $\widehat G$ equipped with the
uniformity $\mu_{\sB_\kappa}$.
The proof of the following result is a direct consequence of Theorem \ref{th_alfa_peso}.

\begin{Proposition}
$\mu_{\mathcal B_\kappa}\widehat G$ is $\kappa$-narrow.
\begin{proof} It is enough to observe that $\mu_{\mathcal B_\kappa}\repng$ is $\kappa$-narrow
for all $n<\omega$.
\end{proof}
\end{Proposition}

\section{Determined groups}
This section is dedicated to apply the results obtained previously
to a question that has attracted the attention of workers interested in
duality theory. Let us assume for the moment that $G$ is an abelian group so that
$\widehat G=\hbox{irrep}(G)=\hbox{irrep}_1(G)$ is also a topological group.
When $D$ is a dense subgroup of a topological abelian group $G$,
then the restriction homomorphism $R|_{D}:\widehat{G}\to \widehat{D}$ of
the dual groups is a continuous isomorphism, but need not be a
topological isomorphism. According to Comfort, Raczkowski and
Trigos-Arrieta \cite{comractri:04}, a (dense) subgroup $D$
of a topological abelian group $G$ {\em determines} $G$ if the
homomorphism $R|_{D}:\widehat{G}\to \widehat{D}$ is a topological
isomorphism. If every dense subgroup of $G $ determines it, then
$G$ is called {\em determined}. The cornerstone in this topic is the
following theorem due to  Aussenhofer \cite{aus} and, independently,
Chasco~\cite{Chasco}:

\begin{Theorem}[Au$\ss$enhofer, Chasco]\label{CA}
A metrizable abelian group  $G$ is  determined.
\end{Theorem}

Comfort, Raczkowski and and Trigos-Arrieta \cite{comractri:04} noticed
that this theorem fails for non-metrizable groups $G$ even when
$G$ is compact. More precisely, they proved that every
non-metrizable compact  group $G$ of weight $\geq \frak c$  contains
a dense subgroup that does not determine $G$. Hence, under the
assumption of the continuum hypothesis, every determined compact
group $G$ is metrizable. Subsequently, it was shown in \cite{hermactri}
that the result also holds without assuming the continuum hypothesis.
Furthermore, Dikranjan and Shakmatov
\cite{dik_sha:jmaa} proved that, for a compact abelian group $G$,
no subgroup with cardinality
smaller than $\W(G)$ may determine it. However, this result does not extend to non-compact
groups.

\begin{Example}
Let $G=\R^+$ be the group of reals equipped with the Bohr topology; that is, the topology
of pointwise convergence on the elements of the dual group that, incidently,
coincides with the group $\R$ equipped with the standard topology. We have the following canonical isomorphisms

$$\widehat{ \R^+}\cong \widehat \R\cong \R\cong \widehat \Q\cong
\widehat{\Q^+}.$$

This means that the dense subgroup of rational numbers determine $G$ in spite that $$\W(\R^+)=|\R|=\frak c\gneq |\Q|.$$
\end{Example}

Separately, in  \cite{lukacs:determined}, Luk\'acs considered the first approach
to non-abelian determined groups. He extended the Au{\ss}enhofer-Chasco
theorem mentioned above by proving the following: if $G$ is a metrizable group, $H$ is a dense subgroup,
and $K$ is a compact Lie group, then the spaces $CHom(G,K)_{\mathcal K(G)}$ and $CHom(H,K)_{\mathcal K(H)}$
are homeomorphic. Our approach to this question is different.

Suppose that $G$ is a compact metrizable abelian group and $H$ is dense in $G$.
From Au$\ss$enhofer-Chasco theorem, we deduce the existence of a compact subset
$K$ of $H$ such that $$K^\rhd=\{\chi\in \widehat G : \|\chi (x)-1\|\leq\sqrt 2\,\,\forall x\in K \}=\{e\}.$$
This means that for all $\chi, \rho$ in $\widehat G$ such that
$$\sup \{ |\chi(x)-\rho(x)| :\ \forall x\in K \}\leq \sqrt 2$$
we have $\chi=\rho$.

Therefore, a single compact subset $K$ of $H$ equips $\widehat G$ with the discrete topology
and the restriction mapping $$R|_{H}:\widehat G\longrightarrow \widehat H$$
is not only a homeomorphism but a \emph{topological isomorphism}.
Our goal is to extend this  result for non-abelian groups. Notwithstanding this,
we notice that the results of this section are also new for abelian groups.

Let $G$ be a topological group 
and let $\mu_{\mathcal K}\ir$ denote  the set $\ir$ equipped with
the uniformity generated by the bornology $\mathcal K(G)$ on $G$. If
$\pi : \ir \longrightarrow \widehat G$ is the canonical quotient mapping,
then $\mu_{\mathcal K}\widehat G$ designates the dual object $\widehat G$
equipped with the final uniformity $\pi[\mu_{\mathcal K}]$.
\begin{Definition}
{\rm We say that a subgroup $H$ of a group $G$ \emph{determines} $G$ when the restriction
mapping $R|_{H}: \mu_{\mathcal K}\widehat G \longrightarrow \mu_{\mathcal K}\widehat H$
is an isomorphism of uniform spaces. $G$ is \emph{determined} if every dense subgroup of $G$ determines $G$.}
\end{Definition}

The main result we are concerned with here is the following:\bigskip

\emph{A compact group is determined if and only if is metrizable\bigskip}

The proof of this result (as we have approached it) lies on the structure
theory of compact groups and some intricate properties of their continuous representations
(sufficiency), and the methods developed in the previous sections (necessity).
Next we are going to introduce the basic definitions for non-abelian groups and
we will prove that every determined compact group is metrizable. We leave
the proof that every metrizable compact group is determined
for a subsequent paper \cite{ferher:ii}.

In the sequel, we are going to present some necessary conditions for a group to be determined.
Previously, we need the following definition.

\begin{Definition}
The \emph{compact weight} $\W_k(X)$ of a topological space $X$ is
the cardinal number
 $\W_k(X)=\sup\{\W(K)\,:\,K\text{ is compact in }G\}$
\end{Definition}

\begin{Proposition}\label{pro_determined}
Let $G$ be a topological group. If $H$ is a dense subgroup of $G$
such that $|H|\leq\kappa$, then
$\mu_{\mathcal K}\widehat H$ is $\kappa$-narrow.
\end{Proposition}
\begin{proof} If $|H|\leq\kappa$, then $\W(C)\leq\kappa$ for all $C\in \mathcal K(H)$. This
means that $\mu_{\mathcal K}\hbox{irrep}_n(H)$ is $\kappa$-narrow for all $n<\omega$. Therefore
the quotient uniform space $\mu_{\mathcal K}\widehat H_n$ must also be $\kappa$-narrow
for all $n<\omega$. This completes the proof.

\end{proof}

As a consequence we obtain

\begin{cor}\ Let $G$ be a topological abelian group. If $H$ is a dense subgroup of $G$
such that $\kappa=|H|<\W_k(G)$, then $H$ does not determine $G$.
\end{cor}
\begin{proof}
For abelian groups $G$, we have $\ir=\hbox{irrep}_1 (G)=\widehat G$. Thus
$\mu_{\mathcal K}\widehat H$ is $\kappa$-narrow. On the other hand,
since $\kappa<\W_K(G)$, it
follows that
$\mu_{\mathcal K}\ir=\mu_{\mathcal K}\widehat G$ may not be $\kappa$-narrow by  Theorem \ref{th_alfa_peso}.
This completes the proof.
\end{proof}

\begin{cor}\label{co_subgroup}
Let $G$ be a compact group. If $H$ is a dense subgroup of $G$
such that $\kappa=|H|<\W(G)$, then $H$ does not determine $G$.
\end{cor}
\begin{proof}
Since $G$ is compact, it follows that $\mu_{\mathcal K}\widehat G_n$ is uniformly
discrete for all $n<\omega$ (see \cite{heyer70}). On the other hand, by Peter-Weyl Theorem, we know that
$\W(G)=|\widehat G_N|$ for some natural $N$. Combining this fact with Theorem \ref{th_alfa_peso}, we obtain that
$\mu_{\mathcal K}\widehat G_N$ may not be $\kappa$-narrow. Now, $\kappa=|H|$ implies
that $\mu_{\mathcal K}\widehat H_n$ is $\kappa$-narrow for all $n<\omega$.
This completes the proof.
\end{proof}

As a consequence, we obtain.
\begin{cor}
Let $G$ be a compact or a topological abelian group. If $2^{|H|}<|G|$ then $H$ does not determine $G$.
\end{cor}

In order to prove that every compact determined group is metrizable we need the following useful fact about quotients of
determined compact groups.

\begin{Lemma}\label{le_quotient}
Every quotient of a determined compact group is a determined compact group.
\begin{proof} Let $G$ be a determined compact group and let $\pi: G\longrightarrow G/N$ be a canonical quotient map,
where $N$ is a closed normal subgroup of $G$. Given an arbitrary dense subgroup $H$ of $G/N$, we set $L=\pi^{-1}(H)$
that  is a dense subgroup of $G$. Firstly, observe that the map
$i_1:\mu_{\mathcal{K}} \hbox{irrep}_n(H)\longrightarrow \mu_{\mathcal{K}}\hbox{irrep}_n(L)$,
defined by $i_1(\varphi)=\varphi\circ\pi$ for all $\varphi\in \irns (H)$, is a uniform embedding
 (i.e., a uniform isomorphism onto its image). Indeed, it is clear that $i_1$ is one-to-one. On the other hand,
 if $C\in \mathcal{K}(L)$ then $\pi(C)\in \mathcal{K}(H)$ and
 $(i_1\times i_1)(B_{H}(\pi(C),\epsilon )=B_L(CN,\epsilon)\cap i_1(\hbox{irrep}_n(H))^2\subseteq B_L(C,\epsilon)$.
 Therefore $i_1$ is uniformly continuous. Conversely, if $K\in \mathcal{K}(H)$, since  $\pi^{-1}(K)\in \mathcal K(L)$,
 it follows that $(i_1\times i_1)^{-1}(B_L(\pi^{-1}(K),\epsilon)\cap i_1(\hbox{irrep}_n (H))^2=B_{H}(K,\epsilon)$.
 Therefore $i_1$ is a uniform embedding as claimed.

Now, let us verify that the map
$\widehat{i}_1:\mu_{\mathcal{K}} \widehat{H}_n\longrightarrow \mu_{\mathcal{K}}\widehat{L}_n$,
defined by $\widehat{i}_1([\varphi])=[\varphi\circ\pi]$ for all $\varphi\in \irns (H)$, is also a uniform embedding.
Indeed, consider the diagram
\[
\xymatrix{  \mu_{\mathcal{K}} \irns (H) \ar@{>}[rr]^{Q_n^{(1)}} \ar@{>}[d]^{i_1} & &
\mu_{\mathcal{K}}\widehat{H}_n \ar@<1ex>[d]^{\widehat{i}_1} \\
\mu_{\mathcal{K}} \irns (L) \ar@{>}[rr]_{Q_n^{(2)}} & &  \mu_{\mathcal{K}}\widehat{L}_n }
\]
\medskip

where $Q_n^{(i)}$, $1\leq i\leq 2$, are quotient maps. Then $\widehat{i}_1\circ Q_n^{(1)}=Q_n^{(2)}\circ i_1$ is
uniformly continuous by the commutativity of the diagram. As a consequence, $\widehat i_1$ is uniformly continuous
because it is defined on a quotient (uniform) space. On the other hand, since $G$ is determined, the restriction map
$R|_{L}:\mu_{\mathcal{K}}\widehat{G}_n :\longrightarrow \mu_{\mathcal{K}}\widehat{L}_n$ is a uniform isomorphism
and we have that $\mu_{\mathcal{K}}\widehat{L}_n$ is uniformly discrete because $\mu_{\mathcal{K}}\widehat{G}_n$
is (see \cite{heyer70}). Therefore $\widehat{i}_{1}^{-1}|_{\widehat{i}_1(\widehat H_n)}$ is also
necessarily uniformly continuous.
In like manner, the map
$\widehat{i}_2:\mu_{\mathcal{K}} (\widehat{G/N})_n\longrightarrow \mu_{\mathcal{K}}\widehat{G}_n$ is a uniform embedding
as well. Now,  we have the following commutative diagram
\[
\xymatrix{  \mu_{\mathcal{K}} (\widehat{G/N})_n \ar@{>}[rr]^{R|_{H}} \ar@{>}[d]^{\widehat i_2} & & \mu_{\mathcal{K}}
\widehat H_n \ar@<1ex>[d]^{\widehat i_1} \\
\mu_{\mathcal{K}} \widehat G_n \ar@{=}[rr]_{R|_{L}} & &  \mu_{\mathcal{K}} \widehat L_n }
\]

\noindent where $R|_{H}=\widehat i_1^{-1}|_{\widehat{i}_1(\widehat H_n)}\circ R|_{L}\circ \widehat i_2$ is a uniform isomorphism.
This completes the proof.
\end{proof}
\end{Lemma}

Next follows the main result in this section.
\begin{Theorem}
Every determined compact group is metrizable.
\end{Theorem}
\begin{proof}
Let $G$ be a compact group of uncountable weight. By the Peter-Weyl Theorem,
$\widehat G_n$ must be uncountable for some $n\in \omega$. Fix a set $\Gamma$ of
$n$-dimensional irreducible continuous representations with
$|\Gamma|=\aleph_1$. If $N = \bigcap \{ \ker (\varphi) : \varphi\in \Gamma \}$,
then the quotient group $G/N$ has weight $\aleph_1$ becausee it is embedded in $\mathbb U(n)^{\aleph_1}$.
It is well known that $G/N$ contains a countable dense subgroup $H$ (see \cite{comfort:handbook}).
By Lemma \ref{co_subgroup}, the subgroup $H$ does not determine $G/N$ and by Lemma \ref{le_quotient}
this means that $G$ may not be a determined group.
\end{proof}

\end{document}